\documentclass[conference]{IEEEtran}

\usepackage{cite}
\usepackage{amsmath,amssymb,amsfonts}
\usepackage{algorithmic}
\usepackage{graphicx}
\usepackage{textcomp}
\usepackage{xcolor} 
\usepackage{epsfig}
\usepackage[font=small,skip=0pt]{caption}
\usepackage{amsmath,amssymb,amsthm}
\usepackage{comment} 
\usepackage{tabularx}
\usepackage{graphicx}
\usepackage{enumitem}
\usepackage{eufrak}
\usepackage{lipsum}
\usepackage{stfloats}
\usepackage{bbm}
\usepackage{relsize}
\usepackage{multicol}
\usepackage{enumitem}

\usepackage[ruled]{algorithm2e}
\usepackage[bindingoffset=0.2in, left=.55 in, right=.75 in,  top=.75in, bottom=.8in, footskip=.25in]{geometry}

\newtheorem{thm}{Theorem}

\newtheorem{propos}{Proposition}

\DeclareMathOperator{\erfc}{erfc}

\def\bsum{\mathlarger{\sum}}

\def\BibTeX{{\rm B\kern-.05em{\sc i\kern-.025em b}\kern-.08em
		T\kern-.1667em\lower.7ex\hbox{E}\kern-.125emX}}
\begin{document}
	
	\title{	\vspace{0.25 in}
		A New Class of Path-Following Method for Time-Varying Optimization with Optimal Parametric Functions
	}
	\author{	\IEEEauthorblockN{Mohsen Amidzadeh, \textit{Member, IEEE}\\
		 Department of Computer Science, Aalto University, Finland\\
		 mohsen.amidzade@aalto.fi}
	\thanks{M. Amidzadeh is with the Department of Communication and Networking, Aalto University, Finland
	(mohsen.amidzade@aalto.fi)	}
	}
	\maketitle

	\begin{abstract}
		In this paper, we consider a formulation of nonlinear constrained optimization problems. 
		We reformulate it as a time-varying optimization using continuous-time parametric functions 
		and derive a dynamical system for tracking the optimal solution. 
		We then re-parameterize the dynamical system to express it based on a linear combination of the parametric functions. 
		Calculus of variations is applied to optimize the parametric functions, 		
		so that the optimality distance of the solution is minimized. 	
		Accordingly, an iterative dynamic algorithm, named as OP-TVO, 
		is devised to find the solution with an efficient convergence rate. 
		We benchmark the performance of the proposed algorithm with the prediction-correction method (PCM)
		from the optimality and computational complexity point-of-views.
		The results show that  OP-TVO can compete with PCM
		for the optimization problem of interest,
		which indicates it can be a promising approach to replace PCM for some time-varying optimization problems.
		Furthermore, this work provides a novel paradigm for solving parametric dynamical system.
	\end{abstract}

	
	
	\section{Introduction}\label{SecIntro}
	Time-Varying Optimization (TVO) problems pertain to parametric optimization problems
	when the objective, constraints, or both are parameterized using continuously-varying functions.
	This paradigm  is exploited to find the optimal trajectory of solutions in continuous-time optimization challenges.
	In addition, for optimization problems for which the optimal solution is known for one specific configuration, 
	TVO can be used to extrapolate the solution to  the settings  of interest.
	
	TVO is studied in the context of parametric programming \cite{Guddat1990, Dinh2012, Kungurtsev2014, Kungurtsev2017, Liao2018}, 
	where the optimization problem is parametrized using continuous parameters.
	In \cite{Kungurtsev2014}, the authors develop prediction-correction methods to solve nonlinear constrained TVO problems.
	Their approach tracks a trajectory path with some convergence guarantees.
	A path-following procedure has been presented in \cite{Kungurtsev2017} 
	to trace a solution path of a parametric nonlinear problem.
	The authors utilize quadratic programming as a tracking procedure 
	and derive some convergence properties for their approach.
	In \cite{Liao2018}, the authors leverage a path-following method 
	to track the solution of parametric nonlinear constrained programs
	using a semi-smooth barrier function.
	
	TVO can be also considered as an extension of time-invariant optimization problems 
	\cite{Simonetto2016, Simonetto2017, Fazlyab2018, Simonetto2019, Simonetto2020}.
	In \cite{Simonetto2016}, a discrete-time prediction-correction approach has been proposed to minimize 
	unconstrained time-varying functions.
	They analyze the asymptotic tracking error to ensure convergence of the solution.
	In \cite{Simonetto2017}, the authors present prediction-correction methods to track the optimal solution trajectory, 
	in the primal space, with a bounded asymptotical error.
	Then, in \cite{Fazlyab2018}, an interior-point method is developed for optimization problems 
	with time-varying cost and constraint functions. 
	The authors formulate a continuous-time dynamical system to track the optimal solution 
	with asymptotical tracking error being bounded.
	In \cite{Simonetto2019}, prediction-correction methods have been devised, in the dual space, to track the solutions 
	of time-varying linearly constrained problems. 
	
	When it comes to the approach for solving TVO problems, the prediction-correction schemes 
	as promising tracking algorithms, should be addressed \cite{Simonetto2016, Simonetto2017,Guo2018, Simonetto2019, Paternain2019, Bastianello2019, Bernstein2019}.
	It constitutes a dynamic-tracking procedure or predictor step to trace the solution trajectory over time,
	and a Newton-based iterative method or corrector step to adjust the error of predictions.
	Substantial efforts have been made to analyze the convergence of
	prediction-correction approaches and ensure the boundedness of tracking error \cite{Simonetto2017,Fazlyab2018,Guo2018}.
	
	In this paper, we study a class of nonlinear constrained optimization problems whose optimal solution is known 
	for a setting, and the optimal solution is aimed for a target configuration. 
	We then exploit the notion of TVO to reformulate the problem based on parametric programming. 
	We then explore a set of parametric functions such that the optimality distance is optimized. 
	In contrast to the works based on the prediction-correction methods, 
	we focus on a functional optimization problem to expedite the convergence rate. 
	In other words, this paper differs from prediction-correction approaches, 
	in the sense that the parametric functions are designed to minimize the optimality distance of the solution 
	instead of correcting the predictions.
	
	The main contributions of this paper are listed as follows:	
	\begin{itemize}[leftmargin=*]
		\item We re-express a class of nonlinear constrained optimization problems as 
		a time-varying optimization problem with continuous-time parametric functions. 
		We then use a re-parametrization trick to represent the corresponding dynamical system
		based on a linear combination of the parametric functions.
		
		\item We then globally optimize the parametric functions using a functional optimization problem, 
		and devise an iterative algorithm with the optimality distance of the solutions being minimized.
		We call the devised algorithm Optimal Parametric Time-Varying Optimization (OP-TVO).
		
		\item We compare OP-TVO with a prediction-correction method from the literature, 
		from the optimality and computational complexity perspectives.
	\end{itemize}

	\textbf{Notations}: In this paper, we use lower-case $a$ for scalars, 
	bold-face lower-case $\textbf{a}$ for vectors 
	and bold-face uppercase $\textbf{A}$ for matrices. 
	Further, $\textbf{A}^\top$ is the transpose of $\textbf{A}$,
	$\|\textbf{a}\|$ is the euclidean norm of $\textbf{a}$,
	$\boldsymbol{\nabla}_{\textbf{a}} g(\cdot)$ and $\boldsymbol{\nabla}^2_{\textbf{a}} g(\cdot)$ 
	are the gradient vector and Hessian matrix 
	of multivariate function $g(\textbf{a})$ with respect to (w.r.t.) vector $\textbf{a}$, respectively.
	We show the components of a $n$-dimensional column vector $\boldsymbol{a}$ using the notation 
	$\boldsymbol{a}=[a_1,\ldots,a_n]^\top$. 
	Further, $\{a_n\}_1^N$ collects the components of vector $\boldsymbol{a}$ from $n=1$ to $n=N$.
	We use $\textbf{I}$, $\textbf{1}$ and $\textbf{0}$ to denote the identity matrix, 
	all-ones and all-zeros vectors, respectively.
 	We use $\dot{\boldsymbol{a}}(\theta )$ to represent the derivative of $\boldsymbol{a}(\theta)$ 
	w.r.t. $\theta$.
	
	\section{Problem Statement}\label{SecModel}
	This paper considers a class of nonlinear constrained optimization problems.
	The problem includes an objective function $f(\cdot):\mathbb{R}^N \rightarrow \mathbb{R}$,
	vector-valued constraint functions $\boldsymbol{h}_m(\cdot) = [h_{m,1},\ldots,h_{m,N}]^\top(\cdot):\mathbb{R}^N \rightarrow \mathbb{R}^N$
	and optimization variables $\boldsymbol{x}_m \in \mathbb{R}^N$ for $m\in\{1,\ldots,M\}$.
	The problem under study is:
	\begin{align}
		P_0:\qquad  &\min_{ \{\boldsymbol{x}_m \}_1^M }~ \sum_{m=1}^M a_m f(\boldsymbol{x_m}) \notag \\
		&~~~\mbox{s.t.}\quad\sum_{m=1}^M \boldsymbol{h}_m(\boldsymbol{x}_m) = \boldsymbol{u},
	\end{align}
	where $\boldsymbol{u} \in \mathbb{R}^N$.
	We assume that the optimal solution of Problem $P_0$ is known for the given non-zero parameters $a_m=p_{0,m}$, 
	with $m \in \{1,\ldots,M\}$. 
	However, the optimal solution is unknown and desired for the non-zero target parameters $a_m=p_{\tau,m}$.

	Note that this type of optimization problem arises in distributed optimizations or multi-agent systems
	where distinct agents produce agent-specific rewards and together make an overall cost function.
	The aim thus is obtaining an agent-specific variables $\boldsymbol{x}_m$ optimizing this overall cost function.
	This also may arise in constrained problems with an objective being established from different cost functions with distinct weights $\{a_m\}_1^M$. 
	
	The goal is to express $P_0$ based on a TVO problem with parametric functions,
	and devise a path-following method with convergence rate being optimized.
	To optimize it, we follow a functional optimization approach to design the parametric functions.
	As such, we parameterize $\{a_m\}_1^M$ with the parametric functions $\{b_m(\theta)\}_1^M$ 
	and parameter $\theta \in \mathbb{R}^+ \cup \{0\}$, so that 
	\begin{align}\label{EQ_limit}
	\!\! \lim_{\theta \rightarrow 0} b_m(\theta) = p_{0,m}, \qquad \lim_{\theta \rightarrow \tau} b_m(\theta) = p_{\tau,m},
	\end{align}
	for $m\in \{1,\ldots,M\}$. 
	Note that such parametric functions as presented in \eqref{EQ_limit} are not unique
	and we aim to find the optimal function throughout this work.

	We  then consider the following TVO problem
	\begin{align}
		P_1(\theta):\qquad \boldsymbol{x}^*(\theta)=\: &\underset{\{\boldsymbol{x}_m(\theta) \}_1^M}{\mathrm{argmin}}~ \sum_{m=1}^M b_m(\theta) f\big( \boldsymbol{x_m}(\theta) \big) \notag \\
		&~~~\mbox{s.t.}\quad \sum_{m=1}^M \boldsymbol{h}_m\big( \boldsymbol{x}_m(\theta) \big) = \boldsymbol{u},
	\end{align}
	where $\boldsymbol{x}(\theta) =[\boldsymbol{x}_1^\top,\ldots,\boldsymbol{x}_M^\top]^\top(\theta)$.
	As declared, we assume that the solution of $P_1(\theta)$ for $\theta=0$ is given, 
	and the solution at target $\theta=\tau>0$ is to be found.		
	
	A naive approach to find the solution of $P_1(\theta)$ is to use a Newton-based iterative algorithm.
	However, this approach suffers from low convergence rate
	and its solution optimality depends on a step-length parameter.
	
	Instead, we develop a time-varying approach to exploit the information of optimal solution of $P_1(0)$.
	We thus formulate a continuous-time dynamical system 
	\begin{align}\label{EQ_dynamic}
	\dot{\boldsymbol{x}}_m(\theta) = \boldsymbol{\phi}_m\big( \boldsymbol{x}(\theta),\theta \big)\!:~\mathbb{R}^{NM}\times\mathbb{R}^+\cup \{0\} \to \mathbb{R}^N,
	\end{align}
	with optimal trajectory solution denoted by $\boldsymbol{x}^*(\cdot)$.
	We then devise an iterative approach to predict  $\boldsymbol{x}^*(\cdot)$ by  $\boldsymbol{x}(\cdot)$
	so that the optimality distance $\| \boldsymbol{x}(\theta) - \boldsymbol{x}^*(\theta) \|$ 
	is minimized for $\theta \to \tau$.
	Note that \eqref{EQ_dynamic} shows a set of ODEs which should be solved with
	initial condition $\boldsymbol{x}(0)$ to give the target point $\boldsymbol{x}(\tau)$.
	As such, we intend to jointly solve TVO $P_1(\theta)$ and design $\{b_m(\theta)\}_1^M$ for $\theta\in [0,\tau]$
	such that the optimality distance is optimized.
	
	Let the following assumptions hold for TVO $P_1(\theta)$.
	
	
	\textbf{Assumption I:} 
	The objective function $f(\cdot)$ and constraints $\boldsymbol{h}_m(\cdot)$ are twice continuously differentiable 
	with respect to (w.r.t.) $\boldsymbol{x}_m$.
	
	\textbf{Assumption II:} 
	The matrices $\{\boldsymbol{\mathcal{J}}_m \}_1^M$ are invertible for $\theta \in [0,\tau]$,
	where $\boldsymbol{\mathcal{J}}_m \in \mathbb{R}^{N\times N}$ is the transpose of Jacobian matrix 
	of $\boldsymbol{h}_m(\cdot)$
	w.r.t. $\boldsymbol{x}_m(\theta)$.
	
	\section{ ODEs associated with TVO $P_1(\theta$) }
	\begin{propos}
	The	solution of Karush–Kuhn–Tucker conditions of problem $P_1(\theta)$, for $\theta\in [0,\tau]$, 
	can be found by the pair $(\boldsymbol{x}(\theta),\boldsymbol{\lambda})$ which follows the dynamical system \eqref{EQ_dynamic} with:
	\begingroup\makeatletter\def\f@size{9.5}\check@mathfonts
	\begin{align}\label{EQ_ODE1}
	&\boldsymbol{\phi}_m\big( \boldsymbol{x}(\theta),\theta \big) =\notag \\
	&~- \Big( b_m(\theta) \nabla_m^2 f + \sum_{n=1}^N \lambda_n \nabla^2_m h_{m,n} \Big)^{-1}
	\boldsymbol{\mathcal{J}}_m \:\Big(\boldsymbol{\dot{\lambda}} - \frac{\dot{b}_m(\theta)}{b_m(\theta)} \boldsymbol{\lambda} \Big),
	\end{align}
	\endgroup
	where 
	\begin{align}\label{EQ_lambda}
		\boldsymbol{\lambda} = -b_m(\theta) \boldsymbol{\mathcal{J}}_m^{-1} \nabla_m f, ~~~~~{\rm for}~~m\in \{1,\ldots,M\},
	\end{align}
	and $\boldsymbol{\dot{\lambda}} $ is obtained using
	\begin{align}\label{EQ_constraint}
	\sum_{m=1}^M \boldsymbol{\mathcal{J}}_m^\top \: \dot{\boldsymbol{x}}_m(\theta) = \boldsymbol{0}.
	\end{align}
	\end{propos}
	\begin{proof}
	Please refer to Appendix \ref{App1}.
	\end{proof}
	According to \eqref{EQ_ODE1}, the dynamical system $\{\boldsymbol{\phi}_m(\cdot,\cdot) \}_1^M$ has been formulated 
	based on a nonlinear combination of two parametric functions, ${b}_m(\theta)$ and $\dot{b}_m(\theta)$. 
	In the sequel, we use a decomposition trick to express the dynamical system as a linear combination of parametric functions.

	\subsection{Reparameterizing based on a Decomposition}
	We introduce the parametric functions  
	\begin{align}\label{EQ_reparamFunc}
	c_m(\theta):=\dfrac{\dot{b}_m(\theta)}{b_m(\theta)}, \qquad m\in \{1,\ldots,M\},
	\end{align}
	which should satisfy $ \int_0^\tau c_m(\theta)d\theta = \log\Big( \dfrac{p_{\tau,m}}{p_{0,m}} \Big):=\psi_m$
	based on \eqref{EQ_limit}.
	We then have:
	\begin{thm}
	The dynamical system \eqref{EQ_ODE1} can be re-parameterized based on 
	the following linear combination of $\{c_m(\theta)\}_1^M$:
	\begin{align}\label{EQ_ODE_linear}
		 \dot{\boldsymbol{x}}(\theta) = \boldsymbol{\phi}\big( \boldsymbol{x}(\theta),\theta \big) = \boldsymbol{\Gamma}(\theta) \: \boldsymbol{c}(\theta),
	\end{align}
	where $\boldsymbol{\phi}(\cdot,\cdot) =[\boldsymbol{\phi}_1^\top,\ldots,\boldsymbol{\phi}_M^\top]^\top(\cdot,\cdot)$, 
	\begin{align*}
	\boldsymbol{\Gamma}(\theta) &= 
	\begin{pmatrix}
		\boldsymbol{\gamma}_{11} & \ldots & \boldsymbol{\gamma}_{1M} \\	
		\vdots&		\vdots&		\vdots \\
		\boldsymbol{\gamma}_{N1} & \ldots & \boldsymbol{\gamma}_{NM}
	\end{pmatrix} ~ \in ~ \mathbb{R}^{NM\times M},\\
	\boldsymbol{\gamma}_{nm} &= -G_n^{-1}\bigg( \Big( \sum_{k=1}^M D_k\Big)^{-1} D_m - \delta_{nm} \boldsymbol{I} \bigg)\boldsymbol{1},\\ 
	G_n &= {\rm diag}(\boldsymbol{v})^{-1} \boldsymbol{\mathcal{J}}_n^{-1}\: \bigg(  \nabla_n^2f + \sum_{j=1}^N  v_j \nabla_n^2h_{n,j} \bigg),\\
	D_m &= \boldsymbol{\mathcal{J}}_m^\top \: G_m^{-1},
	\end{align*}
	for $n \in \{1,\ldots,N\}$ and $m \in \{1,\ldots,M\}$ 
	with 
	\begin{align}\label{EQ_nu}
	\boldsymbol{v} = -\boldsymbol{\mathcal{J}}_m^{-1}\: \nabla_m f.
	\end{align}
	\end{thm}
	\begin{proof}
		We introduce $N$-by-$N$ matrices $\{G_m\}_1^M$ and exploit the following decomposition: 
		\begin{align}\label{EQ_decompose}
			b_m(\theta) \nabla_m^2 f + \sum_{n=1}^N \lambda_n \nabla^2_m h_{m,n} = \boldsymbol{\mathcal{J}}_m\: {\rm  diag}(\boldsymbol{\lambda})G_m.
		\end{align}
		Then, we get:
		\begin{align}\label{EQ_G}
		G_m = {\rm diag}(\boldsymbol{v})^{-1} \boldsymbol{\mathcal{J}}_m^{-1}\:\Big( \nabla_m^2 f + \sum_{n=1}^N v_n \nabla_m^2 h_{m,n} \Big),
		\end{align}
		for which we used $\boldsymbol{\lambda} = b_m(\theta)\boldsymbol{v}$ based on \eqref{EQ_lambda} and \eqref{EQ_nu}.
		Equation \eqref{EQ_G} shows that $G_m$ is notably independent of parametric function $b_n(\theta)$.
		Now, by plugging \eqref{EQ_decompose} into \eqref{EQ_ODE1}, we obtain:
		\begingroup\makeatletter\def\f@size{9.5}\check@mathfonts
		\begin{align*}
			\dot{\boldsymbol{x}}_n(\theta) = -G_n^{-1}	{\rm diag}(\boldsymbol{\lambda})^{-1} \big( \dot{\boldsymbol{\lambda}}-c_n(\theta)\boldsymbol{\lambda}\big ),
		\end{align*}
		\endgroup
		and according to \eqref{EQ_constraint}, we get:
		\begingroup\makeatletter\def\f@size{9.5}\check@mathfonts
		\begin{align*}
		\boldsymbol{\dot{\lambda}} = \boldsymbol{\lambda}\: \Big(\sum_{k=1}^M D_k \Big)^{-1} \sum_{m=1}^M D_m c_m(\theta) \:\boldsymbol{1},
		\end{align*}
		\endgroup
		which together yields:
		\begingroup\makeatletter\def\f@size{9.5}\check@mathfonts
		\begin{align*}
			\dot{\boldsymbol{x}}_n(\theta) = -G_n^{-1}\bigg( \Big(\sum_{k=1}^M D_k \Big)^{-1} \sum_{m=1}^M D_m c_m(\theta) -c_n(\theta)\boldsymbol{I} \bigg)\boldsymbol{1}.
		\end{align*}
		\endgroup
		Based on the definition of $\boldsymbol{\gamma}_{nm}$, it thus reads:
		\begin{align*} 
		\dot{\boldsymbol{x}}_n(\theta) &= \sum_{m=1}^M \boldsymbol{\gamma}_{nm} c_m(\theta),~~~n\in \{1,\ldots,N\},
		\end{align*}
		which proves the statement.
	\end{proof}

	Note that the dynamical system \eqref{EQ_ODE_linear} depends only on $\boldsymbol{c}(\cdot)$ 
	and not on $\boldsymbol{b}(\cdot)$.
	Moreover, considering that $\boldsymbol{\Gamma}(\cdot)$ does not depend on $\boldsymbol{c}(\cdot)$,
	the dynamical system portrays a linear parametric expression w.r.t. $\boldsymbol{c}(\cdot)$.
	It enables us to find the condition in which solving \eqref{EQ_ODE1} or equivalently \eqref{EQ_ODE_linear} leads to a unique solution.
	Let us first make the following assumptions.
		
	\textbf{Assumption III:} 
	The matrices $\{\nabla_n^2 (f+\boldsymbol{v}^\top\boldsymbol{h}_n)\}_1^N$, $\sum_{k=1}^M D_k$ and $\mbox{diag}(\boldsymbol{v})$ are invertible.
	
	\textbf{Assumption IV:} 
	The derivatives of $\{\nabla_m f\}_1^M$, $\{\boldsymbol{\mathcal{J}}_m \}_1^M$ and $\{\nabla_n^2 (f+\boldsymbol{v}^\top\boldsymbol{h}_n)\}_1^N$ 
	w.r.t. $\{\boldsymbol{x}_m(\theta)\}_1^M$ are bounded.

	\begin{propos}
		If \textbf{Assumptions I, II, III} and \textbf{IV} hold, then the dynamical system \eqref{EQ_ODE_linear} has a unique solution.
	\end{propos}
	\begin{proof}
		By applying Picard–Lindelöf theorem and considering the equivalence between Lipschitz continuity and derivative boundedness, the statement follows.
	\end{proof}
	
	The linear form of \eqref{EQ_ODE_linear} also enables us to design $\boldsymbol{c}(\cdot)$ such that the optimality distance is optimized.
	Having $\boldsymbol{c}(\cdot)$ being designed, 
	we can get: $b_m(\theta)=p_{0,m}\exp\big( \int_0^{\theta} c_m(\xi)d\xi \big)$ 
	based on \eqref{EQ_limit} and \eqref{EQ_reparamFunc}.
	
	\section{Optimality Distance}
	Equation \eqref{EQ_ODE_linear} shows a set of ODEs that is intricate to precisely solve
	due to highly non-linearity w.r.t. $\theta$.
	However, one approach is to use the Euler method \cite{Euler} to approximate $\boldsymbol{x}(\cdot)$ 
	by $\hat{\boldsymbol{x}}(\cdot)$ based on a sequential manner:
	\begin{align}\label{EQ_App}
	\!\!\!\!  \hat{\boldsymbol{x}}(\theta) = \hat{\boldsymbol{x}}(\theta-\Delta\theta) 
	+ \Delta\theta \: \hat{\dot{\boldsymbol{x}}}(\theta-\Delta\theta), \qquad \theta \in (0,\tau],
	\end{align}
	where $\Delta\theta$ is the incremental step 
	and $
	\hat{\dot{\boldsymbol{x}}}(\theta) = \boldsymbol{\phi}\big( \hat{\boldsymbol{x}}(\theta),\theta \big)$.
	For this method, the optimality distance $O_d$ is upper-bounded by:
	\begingroup\makeatletter\def\f@size{9.5}\check@mathfonts
	\begin{align*}
	\!\!\!	O_d = \| \hat{\boldsymbol{x}}(\tau) - \boldsymbol{x}(\tau) \| \leq 
		\frac{ \Delta\theta^2}{2} \sum_{j=1}^L \| \ddot{\boldsymbol{x}}(\tau-j\Delta\theta) \| +
		\mathcal{O}(L\Delta\theta^3),
	\end{align*}
	\endgroup
	where $\tau=L\Delta\theta$.
	This shows that optimality distance is limited by order of $\Delta\theta^2$.
	However, based on \eqref{EQ_App} another upper-bound can be found as follows:
	\begin{align}\label{EQ_OptimalityGap}
		O_d = \left\| \int_0^\tau  \big( \dot{\boldsymbol{x}}(\theta) - \hat{\boldsymbol{m}} \big) d\theta \right\| 
		\leq  \int_0^\tau \| \dot{\boldsymbol{x}}(\theta) - \hat{\boldsymbol{m}} \| d\theta,
	\end{align}
	where
	$$
	\hat{\boldsymbol{m}} = \frac{1}{L} \sum_{j=1}^L \hat{\dot{\boldsymbol{x}}}(\tau-j\Delta\theta).
	$$
	Note that $\hat{\boldsymbol{m}}$ does not depend on $\theta$.
	Consequently, minimizing the upper-bound of $O_d$, i.e., $\int_0^\tau \| \dot{\boldsymbol{x}}(\theta) - \hat{\boldsymbol{m}} \| d\theta$, 
	leads to the optimality distance being minimized.
	
	\section{Optimality Distance Minimization }
	We consider the following functional optimization problem (FOP) to jointly design the parametric functions 
	$\boldsymbol{c}(\cdot)$ and find the optimum solution $\boldsymbol{x}(\cdot)$:
	\begin{align}\label{EQ_J1}
		J_1: \qquad \min_{  \boldsymbol{x}(\cdot), \boldsymbol{c}(\cdot)  }~ & \int_0^\tau  \| \dot{\boldsymbol{x}}(\theta) - \hat{\boldsymbol{m}} \|^2 d\theta \notag \\
		\mbox{s.t.}~~~& \dot{\boldsymbol{x}}(\theta) = \boldsymbol{\Gamma}(\theta) \: \boldsymbol{c}(\theta), \notag \\
		\mbox{s.t.}~~~& \int_0^\tau  \boldsymbol{c}(\theta)d\theta =  \boldsymbol{\psi}.
	\end{align}
	By solving FOP $J_1$, we can achieve the optimal solution $\boldsymbol{x}(\theta)$
	which minimizes the optimality distance 	$O_d$ 
	based on \eqref{EQ_OptimalityGap}.
	To solve \eqref{EQ_J1}, we constitute the Hamiltonian $\mathcal{H}$ as:
	\begingroup\makeatletter\def\f@size{9.5}\check@mathfonts
	$$
	\mathcal{H} =   \| \dot{\boldsymbol{x}}(\theta) - \hat{\boldsymbol{m}} \|^2 
	+\boldsymbol{w}(\theta)^\top \big(\dot{\boldsymbol{x}}(\theta) - \boldsymbol{\Gamma}(\theta) \: \boldsymbol{c}(\theta) \big) 
	+ \boldsymbol{\lambda}^\top \Big( \boldsymbol{c}(\theta)-\frac{1}{\tau}\boldsymbol{\psi} \Big),
	$$
	\endgroup
	where $\boldsymbol{w}(\theta)$ is a co-state variables and $\boldsymbol{\lambda}$ is a Lagrange multiplier.
	Using calculus of variations, the functional solution of \eqref{EQ_J1} is obtained by:
	\begin{equation}
	\begin{aligned}\label{EQ_condJ1}
	\begin{cases}
			\nabla_{\boldsymbol{c}(\theta)} \mathcal{H} = 2\boldsymbol{\Gamma}(\theta)^\top \Big(\boldsymbol{\Gamma}(\theta) \boldsymbol{c}(\theta)-\hat{\boldsymbol{m}}-\frac12\boldsymbol{w}(\theta)  \Big)+\boldsymbol{\lambda}=0\vspace{6pt},\\
			\nabla_{\boldsymbol{x}(\theta)} \mathcal{H} - \frac{d}{d\theta} \nabla_{\dot{\boldsymbol{x}}(\theta)} \mathcal{H} \\
			\qquad ~ \quad = \boldsymbol{w}(\theta)^\top \:\nabla_{\boldsymbol{x}(\theta)} \boldsymbol{\Gamma}(\theta) \:\boldsymbol{c}(\theta)+2\ddot{\boldsymbol{x}}(\theta)+\dot{\boldsymbol{w}}(\theta)=0, \vspace{6pt}\\
			\dot{\boldsymbol{x}}(\theta) - \boldsymbol{\Gamma}(\theta) \: \boldsymbol{c}(\theta) = 0, \qquad \quad
			\int_0^\tau \boldsymbol{c}(\theta)d\theta - \boldsymbol{\psi}=0.
		\end{cases}
	\end{aligned}
	\end{equation}
	The conditions in \eqref{EQ_condJ1} are intricate to solve.
	Further, an estimation of $\boldsymbol{\hat{m}}$ is needed in advance.
	This motivates us to devise an iterative algorithm to find the solution of $J_1$.
	
	In this regard, we perform an iterative algorithm as follows:
	In the beginning, we initialize $\boldsymbol{c}(\cdot)$ such that 
	$\int_0^\tau  \boldsymbol{c}(\theta)d\theta =  \boldsymbol{\psi}$.
	Then, we repeatedly follow two consecutive steps till the algorithm converges.
	These are the \textit{Prediction} and \textit{Parametral-tuning} steps.
	In the prediction step of this iterative approach, 
	we sequentially solve \eqref{EQ_App} using \eqref{EQ_ODE_linear}
	and recently updated $\boldsymbol{c}(\cdot)$, in order to predict $\boldsymbol{x}(\theta)$ for $\theta \in (0,\tau]$. 
	In the parametral-tuning step, we minimize the functional objective 
	$\int_0^\tau \| \dot{\boldsymbol{x}}(\theta) - \hat{\boldsymbol{m}} \|^2 d\theta$ w.r.t. $\boldsymbol{c}(\cdot)$ 
	with $\hat{\boldsymbol{m}}$ being obtained based on the solution of prediction step.
	As declared, we execute these two steps till the convergence.
	We call this algorithm Optimal Parametric Time-Varying Optimization (OP-TVO).
	
    More specifically, we consider the following FOP, in the second step of OP-TVO, to optimize the parametric functions 
    $\boldsymbol{c}(\cdot)$:
	\begingroup\makeatletter\def\f@size{9.75}\check@mathfonts
	\begin{align}\label{EQ_J2}
	\! J_2:\quad  \min_{ \boldsymbol{c}(\cdot)  }~ & \int_0^\tau 
	\Big\| \boldsymbol{\Gamma}(\theta)\:\boldsymbol{c}(\theta)-\hat{\boldsymbol{m}} \Big\|^2 d\theta 
	+  \mu  \int_0^\tau \boldsymbol{c}(\theta)^\top\boldsymbol{c}(\theta) d\theta \notag \\
	\mbox{s.t.}~~& \int_0^\tau \boldsymbol{c}(\theta)d\theta = \boldsymbol{\psi}~,
	\end{align}
	\endgroup
	where the term $\int_0^\tau \boldsymbol{c}(\theta)^\top\boldsymbol{c}(\theta) d\theta$ 
	is additionally added to regulate the smoothness of $\boldsymbol{c}(\theta)$ w.r.t. $\theta$, 
	and $0 < \mu \ll 1$ is the regulation coefficient.
	
	\begin{propos}
	The globally optimal solution of $J_2$ is obtained by:	
	\begin{align}\label{EQ_OptFun2}			
		\boldsymbol{c}(\theta)= \boldsymbol{\Pi}^{-1}(\theta)  
		\left( \boldsymbol{\Gamma}(\theta)^\top \hat{\boldsymbol{m}} -  \boldsymbol{\lambda}\right),
	\end{align}
	where 
	\begin{align*}
		&\boldsymbol{\Pi}(\theta) =  \boldsymbol{\Gamma}(\theta)^\top \boldsymbol{\Gamma}(\theta) + \mu \boldsymbol{I} ,\\
		&\boldsymbol{\lambda} = \left( \int_0^\tau \boldsymbol{\Pi}^{-1}(\theta) d\theta \right)^{-1}
		\left( \int_0^\tau \boldsymbol{\Pi}^{-1}(\theta) \boldsymbol{\Gamma}(\theta)^\top  d\theta \: \hat{\boldsymbol{m}} - \boldsymbol{\psi}  \right).
	\end{align*}
	\end{propos}
	\begin{proof}
		Considering that the dynamical system \eqref{EQ_ODE_linear} has been expressed based on a linear combination
		of parametric functions $\boldsymbol{c}(\cdot)$, $J_2$ is a convex FOP.
		This implies that the globally optimal solution of $J_2$ can be found 
		by applying the Euler-Lagrange equation on $J_2$ \cite{Seierstad1977}. We thus get:
		\begin{align}\label{EQ_Aux1}
		\begin{cases}
			\boldsymbol{\Gamma}(\theta)^\top ( \boldsymbol{\Gamma}(\theta) \: \boldsymbol{c}(\theta) - \hat{\boldsymbol{m}}) + \mu\boldsymbol{c}(\theta) + \boldsymbol{\lambda}&=0,~ \vspace{4 pt}\\
			\int_0^\tau \boldsymbol{c}(\theta)d\theta - \boldsymbol{\psi}&=0~.
		\end{cases}
		\end{align}
		By solving \eqref{EQ_Aux1}, the statement follows.
	\end{proof}
	
		\begin{algorithm}[t]
		\caption{Optimal Parametric Time-Varying Optimization (OP-TVO)}
		\label{Alg_TVO}
		\textbf{Input:} Optimal solution $\boldsymbol{x}(0)$. \\
		\textbf{Outputs:} Optimal solution $\boldsymbol{x}(\tau)$ and parametric functions $\boldsymbol{c}(\cdot)$.\\
		\textbf{Initialization}:\\
		$~~~$ Initialize $\boldsymbol{c}(\cdot)$  so that $\int_0^\tau  \boldsymbol{c}(\theta)d\theta =  \boldsymbol{\psi}$.\\
		$~~~$ Set $\rm flag=1$ and $\rm iter=0$.\\		
		\While{$\rm flag$}{
			iter = iter+1.\vspace{5 pt}\\
			\textbf{Prediction step}:\\
			Predict $\boldsymbol{x}(\theta)$ for $\theta \in (0,\tau]$ using \eqref{EQ_App} and \eqref{EQ_ODE_linear} and based on updated $\boldsymbol{c}(\cdot)$.\\\vspace{5 pt}
			\textbf{Parametral-tuning step:}\\
			Update $\boldsymbol{c}(\cdot)$ using \eqref{EQ_OptFun2} and based on predicted $\boldsymbol{x}(\theta)$ with $\theta \in (0,\tau]$.\vspace{5 pt}\\
			\If{Convergence}{
				$\rm flag=0$.
			}
		}
	\end{algorithm}
	Algorithm \ref{Alg_TVO} shows the pseudo-code of OP-TVO.
	For each iteration (${\rm iter}$), the prediction and parametral-tuning steps are performed
	to jointly predict the optimal solution $\boldsymbol{x}^*(\tau)$ 
	and design the parametric functions $\boldsymbol{c}(\cdot)$.
	We also need a metric for the convergence to stop the algorithm.
	For this, we track the value of $\hat{O}_d := $ $\bsum_{j=1}^L \Big\| \boldsymbol{\phi}\Big( \hat{\boldsymbol{x}}(\tau-j\Delta\theta),\tau-j\Delta\theta \Big)-\hat{\boldsymbol{m}} \Big\|^2 $ 
	as an estimation of the optimality distance.
	We thus consider that the algorithm has converged if the value of $\hat{O}_d$ 
	lies below a threshold $O_{\rm th}$.

	\section{Numerical Results and Discussion}
	To evaluate the devised algorithm OP-TVO, 
	we compare it with a Prediction-Correction Method (PCM) \cite{Simonetto2016},
	as well as with a \textit{Benchmark} solution being obtained using an extremely small incremental step $\Delta\theta=10^{-6}$,
	disregarding its computational complexity.
	Note that, we consider this \textit{Benchmark} as the optimal solution, 
	by which we can compute the optimality distance $O_d$.
	These algorithms have been implemented using Matlab R2022a on a $8\times1.70$ GHz
	Intel Core i5-10310U Processor, equipped with 16 GB of memory and 12 Mbytes of data cache. 
	
	We then consider the constrained optimization problem $E_1$ \cite{Ours_2020,Ours_2022}, 
	and change the constraints to add non-linearity to the problem.
	\begin{align*}
		E_1: \qquad &\min_{\{\boldsymbol{x}_m\}_1^M} ~~\sum_{m=1}^M a_m\erfc\left(\gamma_0 \dfrac{x_{m,1}}{\sqrt{2^{\frac{0.1}{x_{m,2}}}-1}} \right)\\
		&~~\mbox{s.t.} ~~~ \sum_{m=1}^M  \log(1 + x_{m,1}) = L,\\
		&~~\mbox{s.t.} ~~~ \sum_{m=1}^M  x_{m,2}^2 = 1,
	\end{align*}
	where the optimization variables are $\boldsymbol{x}_m  = [x_{m,1},x_{m,2}]^\top$ for $m\in \{1,\ldots,M\}$, 
	$M=100$, $a_m = m^{-\tau}/\sum_{m=1}^M m^{-\tau}$, $\tau=3$ and $\gamma_0=40$.
	For $E_1$, it can be verified that the optimal solution when $a_m=\frac{1}{M}$ is obtained as:
	$$
	   x_{m,1} = \exp(L/M-1), ~~x_{m,2} = \sqrt{1/M},\quad m\in\{1,\ldots,M\}.
	$$
	Therefore, we constitute a TVO problem exactly as $E_1$ 
	but with parametric functions $\{b_m(\theta)\}_1^M$ replacing $\{a_m\}_1^M$ such that:
	$$
	b_m(0) = \frac{1}{M}, \quad b_m(\tau) = a_m, \quad m\in \{1,\ldots,M\}.
	$$
	We further use the re-parametrizing functions $\boldsymbol{c}(\cdot)$
	with $c_m(\theta):=\dfrac{\dot{b}_m(\theta)}{b_m(\theta)}$.
	We then apply Algorithm \ref{Alg_TVO} with hyper-parameters $\mu=10^{-7}$, $O_{\rm th}=10^{-5}$ 
	and  $\Delta\theta=10^{-2}$ to find the optimal solution.
	
	Figures \ref{Fig1} and \ref{Fig2} illustrate the trajectory solutions $\{\boldsymbol{x}_m(\theta) \}_1^M$, 
	being obtained by Algorithm \ref{Alg_TVO}, as a function of $\theta$ for different iterations.
	For the first iteration, the trajectory solutions portray a highly nonlinear behavior.
	However, as the number of iteration increases, this non-linearity reduces.
	When the algorithm converges (3rd iteration), 
	the linear curvature of the trajectory solutions indicates that a precise solution 
	with minimal optimality distance has been obtained.
	\begin{figure}
		\begin{center}
			\includegraphics[width=90 mm]{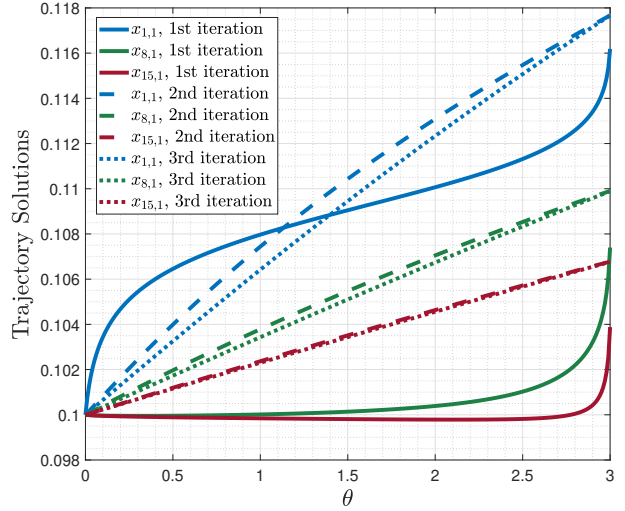}
		\end{center}
		\caption{Trajectory Solutions for $x_{m,1}$ with $m\in \{1,8,15\}$. \label{Fig1}} 
	\end{figure}
	\begin{figure}
		\begin{center}
			\includegraphics[width=90 mm]{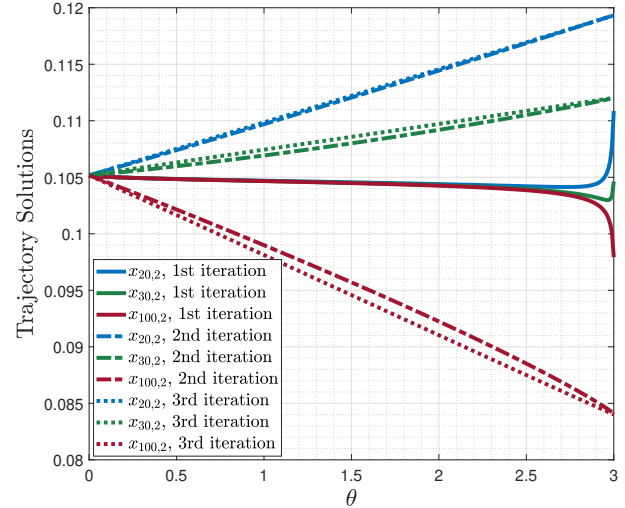}
		\end{center}
		\caption{Trajectory Solutions for $x_{m,2}$ with $m\in \{20,30,100\}$. \label{Fig2}} 
	\end{figure}
	
	We also apply a PCM on Problem $E_1$ to obtain the solution.
	To have a fair comparison, we adjust the hyper-parameter $\Delta\theta$ for PCM 
	such that the corresponding optimal value is almost equal to that of Algorithm \ref{Alg_TVO}.
	As such, we need to decrease it by $1\times10^{-4}$.

	\begin{table*}[t]
	\vspace{10 pt}
	\renewcommand{\arraystretch}{1.6}
	\caption{Performance Result of OP-TVO and PCM.}
	\label{Tab_methods}
	\begin{center}
		\begin{tabular}{c|c c c c c}
			\hline 
			Approach & 
			\begin{tabular}[c]{@{}c@{}} Optimal Value     		\end{tabular} & 
			\begin{tabular}[c]{@{}c@{}} Constraint Violations   \end{tabular} & 
			\begin{tabular}[c]{@{}c@{}} Elapsed  Time  [s] 		\end{tabular} &
			\begin{tabular}[c]{@{}c@{}} $O_d$       			\end{tabular} & 
			\begin{tabular}[c]{@{}c@{}} $\hat{O}_d$          	\end{tabular} \\ [.1ex] \hline
			${\tt Benchmark}$       		 & $5.6089\times10^{-7}$    &1.312$\times 10^{-6}$      & 2066 		&  0.0             & N/A                 \\ \hline
			${\tt \mbox{PCM}}$               & $5.6086\times10^{-7}$    &1.739$\times 10^{-4}$ 		& 271  		& 0.00135  		      	 	& N/A  			     \\ \hline
			${\tt \mbox{OP-TVO},}  $~iter=1  & $3.2443\times10^{-5}$    &0.419  					& 3    		& 0.8291   	 		  	 	& N/A                 \\ 
			${\tt ~~~~~~~~~~~}$~iter=2       & $5.6088\times10^{-7}$    &5.129$\times 10^{-5}$		&52   		& 5.132$\times 10^{-4}$   	& 0.228                \\  
			${\tt ~~~~~~~~~~~}$~iter=3 		 & $5.6088\times10^{-7}$    &5.051$\times 10^{-5}$		&105  		& 5.130$\times 10^{-4}$   	& 2.68$\times 10^{-6}$ \\ 
			${\tt ~~~~~~~~~~~}$~iter=4 		 & $5.6088\times10^{-7}$    &5.050$\times 10^{-5}$		&160  		& 5.130$\times 10^{-4}$   	& 2.14$\times 10^{-6}$ \\ \hline
		\end{tabular}
	\end{center}
	\end{table*}	
	Table \ref{Tab_methods} compares the performance results of PCM, Benchmark and OP-TVO.
	The second column shows the optimal values that these approaches achieve,
	The third column states the extent to which these methods violate the constraints, represented as the summation of constraint violations,
	the fourth column indicates the elapsed time, in seconds, 
	and the fifth and sixth columns show the optimality distance $O_d$ and its estimation $\hat{O}_d$, respectively.
	
	Although PCM achieves a slightly lower objective function value, 
	it is important to highlight that Benchmark notably excels in terms of constraint satisfaction compared to PCM.
	As such, the Benchmark solution is regarded as the reference approach.
	According to the values of $\hat{O}_d$, OP-TVO converges after three iterations.
	It can be seen that OP-TVO with ${\rm iter}=3$ outperforms  PCM
	from the computational complexity perspective 
	as it converges within 105 seconds while PCM converges after 271 seconds.
	Furthermore, the optimality distance $O_d$ of OP-TVO with ${\rm iter}=3$ is superior 
	to the optimality distance of PCM.
	Not to mention that OP-TVO better satisfies the constraints than PCM.
	These imply that OP-TVO gives a more accurate solution compared to PCM,
	and depict that OP-TVO can find the optimal solution with lower computational complexity
	and better performance precision.

	\section{Conclusion}\label{SecConclu}
	In this paper, we reformulated a class of nonlinear constrained optimization problems 
	based on a time-varying optimization with some parametric functions. 
	We then leveraged a re-parametrization trick to find a dynamical system 
	being linearly expressed in terms of the parametric functions.  
	To minimize the optimality distance of the solution, being traced by a dynamical system, 
	we utilized a functional minimization problem. 
	As such, an iterative algorithm, called OP-TVO, was devised to find the solution trajectory with optimal optimality distance. 
	Based on the results, the proposed algorithm outperforms the Prediction-Correction Method (PCM)
	from the optimality distance and convergence rate perspectives.
	The results show that  OP-TVO can be considered as a promising approach to replace PCM for distributed time-varying optimization problems.
	Optimization problems with time-varying constraints can be considered as a future work.

	\section{Acknowledgement}
	This work was partially funded by the Research Council of Finland under grant number 357533.
	
	\bibliographystyle{IEEEtran}
	\bibliography{IEEEabrv,IEEE}

\begin{thebibliography}{10}
\providecommand{\url}[1]{#1}
\csname url@samestyle\endcsname
\providecommand{\newblock}{\relax}
\providecommand{\bibinfo}[2]{#2}
\providecommand{\BIBentrySTDinterwordspacing}{\spaceskip=0pt\relax}
\providecommand{\BIBentryALTinterwordstretchfactor}{4}
\providecommand{\BIBentryALTinterwordspacing}{\spaceskip=\fontdimen2\font plus
\BIBentryALTinterwordstretchfactor\fontdimen3\font minus
  \fontdimen4\font\relax}
\providecommand{\BIBforeignlanguage}[2]{{%
\expandafter\ifx\csname l@#1\endcsname\relax
\typeout{** WARNING: IEEEtran.bst: No hyphenation pattern has been}%
\typeout{** loaded for the language `#1'. Using the pattern for}%
\typeout{** the default language instead.}%
\else
\language=\csname l@#1\endcsname
\fi
#2}}
\providecommand{\BIBdecl}{\relax}
\BIBdecl

\bibitem{Guddat1990}
J.~Guddat, F.~G. Vazquez, and H.~T. Jongen, \emph{Parametric Optimization:
  Singularities, Pathfollowing and Jumps}.\hskip 1em plus 0.5em minus
  0.4em\relax Chichester, U.K.: Wiley, 1990.

\bibitem{Dinh2012}
Q.~T. Dinh, C.~Savorgnan, and M.~Diehl, ``Adjoint-based predictor-corrector
  sequential convex programming for parametric nonlinear optimization,''
  \emph{SIAM J. Optim.}, vol.~22, no.~4, p. 1258–1284, 2012.

\bibitem{Kungurtsev2014}
V.~Kungurtsev and M.~Diehl, ``Sequential quadratic programming methods for
  parametric nonlinear optimization,'' \emph{Computational Optimization and
  Applications,}, vol.~59, no.~3, pp. 475--509, Feb. 2014.

\bibitem{Kungurtsev2017}
V.~Kungurtsev and J.~Jäschke, ``A predictor-corrector path-following algorithm
  for dual-degenerate parametric optimization problems,'' \emph{SIAM J.
  Optim.}, vol.~27, no.~1, p. 538–564, 2017.

\bibitem{Liao2018}
D.~Liao-McPherson, M.~M. Nicotra, and I.~V. Kolmanovsky, ``A semismooth
  predictor corrector method for real-time constrained parametric optimization
  with applications in model predictive control,'' in \emph{2018 IEEE
  Conference on Decision and Control (CDC)}, 2018, pp. 3600--3607.

\bibitem{Simonetto2016}
S.~Andrea, M.~Aryan, K.~Alec, L.~Geert, and A.~Ribeiro, ``A class of
  prediction-correction methods for time-varying convex optimization,''
  \emph{IEEE Transactions on Signal Processing}, vol.~64, no.~17, pp.
  4576--4591, 2016.

\bibitem{Simonetto2017}
A.~Simonetto and E.~Dall’Anese, ``Prediction-correction algorithms for
  time-varying constrained optimization,'' \emph{IEEE Transactions on Signal
  Processing}, vol.~65, no.~20, pp. 5481--5494, 2017.

\bibitem{Fazlyab2018}
M.~Fazlyab, S.~Paternain, V.~M. Preciado, and A.~Ribeiro,
  ``Prediction-correction interior-point method for time-varying convex
  optimization,'' \emph{IEEE Transactions on Automatic Control}, vol.~63,
  no.~7, pp. 1973--1986, 2018.

\bibitem{Simonetto2019}
A.~Simonetto, ``Dual prediction–correction methods for linearly constrained
  time-varying convex programs,'' \emph{IEEE Transactions on Automatic
  Control}, vol.~64, no.~8, pp. 3355--3361, 2019.

\bibitem{Simonetto2020}
A.~Simonetto, E.~Dall'Anese, S.~Paternain, G.~Leus, and G.~B. Giannakis,
  ``Time-varying convex optimization: Time-structured algorithms and
  applications,'' \emph{Proceedings of the IEEE}, vol. 108, no.~11, pp.
  2032--2048, 2020.

\bibitem{Guo2018}
D.~Guo, X.~Lin, Z.~Su, S.~Sun, and Z.~Huang, ``Design and analysis of two
  discrete-time {ZD} algorithms for time-varying nonlinear minimization,''
  \emph{Numer. Algorithms}, vol.~77, no.~1, p. 23–36, 2018.

\bibitem{Paternain2019}
S.~Paternain, M.~Morari, and A.~Ribeiro, ``Real-time model predictive control
  based on prediction-correction algorithms,'' in \emph{2019 IEEE 58th
  Conference on Decision and Control (CDC)}, 2019, pp. 5285--5291.

\bibitem{Bastianello2019}
N.~Bastianello, A.~Simonetto, and R.~Carli, ``Prediction-correction splittings
  for nonsmooth time-varying optimization,'' in \emph{2019 18th European
  Control Conference (ECC)}, 2019, pp. 1963--1968.

\bibitem{Bernstein2019}
A.~Bernstein, E.~Dall'Anese, and A.~Simonetto, ``Online primal-dual methods
  with measurement feedback for time-varying convex optimization,'' \emph{IEEE
  Transactions on Signal Processing}, vol.~67, no.~8, pp. 1978--1991, 2019.

\bibitem{Amidzade_2023}
M.~Amidzadeh, ``Time-varying optimization with optimal parametric function,''
  preprint ArXiv 2210.00931, 2023.

\bibitem{Euler}
A.~Iserles, \emph{A First Course in the Numerical Analysis of Differential
  Equations}, 2nd~ed., ser. Cambridge Texts in Applied Mathematics.\hskip 1em
  plus 0.5em minus 0.4em\relax Cambridge University Press, 2008.

\bibitem{Seierstad1977}
A.~Seierstad and K.~Sydsaeter, ``Sufficient conditions in optimal control
  theory,'' \emph{International Economic Review}, vol.~18, no.~2, pp. 367--391,
  1977.

\bibitem{Ours_2020}
M.~Amidzade, H.~{Al-Tous}, O.~Tirkkonen, and G.~Caire, ``Cellular network
  caching based on multipoint multicast transmissions,'' in \emph{Proc. IEEE
  Glob. Commun. Conf. {(GLOBECOM)}}, Dec. 2020, pp. 1--6.

\bibitem{Ours_2022}
------, ``Cellular traffic offloading with optimized compound single-point
  unicast and cache-based multipoint multicast,'' in \emph{IEEE Wireless
  Commun. and Net. Conf. {(WCNC)}}, 2022, pp. 2268--2273.

\end{thebibliography}
	
\appendix
\subsection{Proof of Proposition 1.}\label{App1}
We constitute the Lagrangian function, and obtain the Karush–Kuhn–Tucker conditions which read:
\begin{align}\label{EQ_APP1}
\begin{cases}
	b_m(\theta) \nabla_m f + \boldsymbol{\mathcal{J}}_m \boldsymbol{\lambda} = \boldsymbol{0}, \vspace{5 pt}\\
	\sum_{m=1}^M \boldsymbol{h_m}(\boldsymbol{x}_m) - \boldsymbol{u}  = \boldsymbol{0},  
\end{cases}
\end{align}
where $\boldsymbol{\lambda}$ is the Lagrange multiplier.
By considering \textbf{Assumption 1} and taking derivative of the recent conditions w.r.t. $\theta$, we get:
\begin{align}\label{EQ_APP2}
\!\!\!\!\!\!
\begin{cases}
	\dot{b}_m(\theta) \nabla_m f +  \Big( b_m(\theta) \nabla_m^2 f + \bsum_{j=1}^N \lambda_j \nabla_m^2 h_{m,j} \Big) \dot{\boldsymbol{x}}_n(\theta) \vspace{-3 pt}\\ 
	\qquad \qquad \quad + \boldsymbol{\mathcal{J}}_m \dot{\boldsymbol{\lambda} }  = \boldsymbol{0}, \vspace{7 pt}\\
	\qquad \sum_{m=1}^M \boldsymbol{\mathcal{J}}_m^\top \dot{\boldsymbol{x}}_m  = \boldsymbol{0} ,
\end{cases}
\end{align}
By combining \eqref{EQ_APP1} and \eqref{EQ_APP2} and considering \textbf{Assumption 2}, the statement follows.
\end{document}